\documentclass[11pt, twoside ]{article}
\usepackage{amsfonts}
\usepackage{amsmath,amsthm}
\usepackage{enumerate}

\usepackage[OT1]{fontenc}
\usepackage{bm}
\usepackage{mathrsfs}

\def\nb{{\Bbb N}}

\def\rb{{\Bbb R}}
\def\cb{{\Bbb C}}

\theoremstyle{plain}
\newtheorem{theorem}{Theorem}
\newtheorem{proposition}{Proposition}
\newtheorem{lemma}{Lemma}

\begin{document}


%
%
%


\title{\textbf{Reformulation of the Li criterion for the Selberg class }}
\author{Kamel Mazhouda\\
\small Faculty of Sciences of Monastir, Department of Mathematics,\\[-4.8pt]
\small  5000 Monastir, Tunisia \\[-4.8pt]
\small E-mail: Kamel.Mazhouda@fsm.rnu.tn
}
\maketitle
\renewcommand{\thefootnote}{}
\footnote{Supported by  the Tunisian-French Grant DGRST-CNRS 14/R 1501}
\



\quad\\[-60pt]
\pagestyle{myheadings}
\markboth{K. Mazhouda}{Reformulation of the Li criterion for Selberg class}

\begin{abstract}
Let $F$ be a function in the Selberg class ${\mathcal S}$ and $a$ be a real number not equal to 1/2. Consider the sum
$$\lambda_{F}(n,a)=\sum_{\rho}\left[1-\left(\frac{\rho-a}{\rho+a-1}\right)^{n}\right],$$
where $\rho$ runs over the non-trivial zeros of $F$. In this paper, we prove that the Riemann hypothesis is equivalent to the positivity of the "modified Li coefficient" $\lambda_{F}(n,a)$, for $n=1,2,..$ and $a<1/2$. Furthermore, we give an explicit arithmetic and asymptotic formula of these coefficients.
\end{abstract}

\section{Introduction}\label{se.1}
The Riemann hypothesis is the subject of several studies and research papers. Most of them provide new reformulations and numerical evidence for this hypothesis . In the literature, there exists various formulations of the Riemann hypothesis. The Li criterion  for the Riemann hypothesis (see. \cite{Li}) is a necessary and sufficient condition that the sequence $$\lambda_{n}=\sum_{\rho}\left[1-\left(1-\frac{1}{\rho}\right)^{n}\right]$$ is non-negative for all $n\in{\nb}$ and where $\rho$ runs over the non-trivial zeros of $\zeta(s)$. This criterion holds for a large class of Dirichlet series so called the  Selberg class as given in \cite{OM1,OM2}. More recently, Omar and Bouanani \cite{O-B} extended the Li criterion
for function fields and established an explicit and asymptotic formula for the Li coefficients.\\

The Selberg class ${\mathcal S}$ \cite{sel}
consists of Dirichlet series
$$F(s)=\sum_{n=1}^{+\infty}\frac{a(n)}{n^{s}}, \qquad Re(s)>1 $$
satisfying the following hypothesis.\\
\begin{itemize}
    \item {\bf   Analytic continuation:} there exists a non negative integer $m$ such \\$(s-1)^{m}F(s)$ is an entire function of finite order.  We denote by  $m_{F}$  the smallest integer  $m$  which satisfies this condition;
    \item {\bf  Functional equation:} for $1\leq j\leq r$, there are positive real numbers $Q_{F},\ \lambda_{j}$ and there are complex numbers $\mu_{j},\ \omega$ with $Re(\mu_{j})\geq0$ and $|\omega|=1$, such that
       $$\phi_{F}(s)=\omega\overline{\phi_{F}(1-\overline{s})}$$
    where
    $$\phi_{F}(s)=F(s)Q_{F}^{s}\prod_{j=1}^{r}\Gamma(\lambda_{j}s+\mu_{j});$$
    \item {\bf Ramanujan hypothesis:} $a(n)=O(n^{\epsilon})$;
    \item {\bf Euler product:} $F(s)$ satisfies
    $$F(s)=\prod_{p}\exp\left(\sum_{k=1}^{+\infty}\frac{b(p^{k})}{p^{ks}}\right)$$
\noindent with suitable coefficients $b(p^{k})$ satisfying $b(p^{k})=O(p^{k\theta})$ for some $\theta<\frac{1}{2}$.
\end{itemize}
 It is  expected that for every function in the Selberg class the analogue of the Riemann hypothesis holds, i.e, that all non trivial (non-real) zeros lie on the critical line $Re(s)=\frac{1}{2}$.
The degree of $F\in{{\mathcal S}}$ is defined by
$$d_{F}=2\sum_{j=1}^{r}\lambda_{j}.$$
 The logarithmic
derivative of $F(s)$  has also the Dirichlet series expression
$$-\frac{F'}{F}(s)=\sum_{n=1}^{+\infty}\Lambda_{F}(n)n^{-s},\qquad Re(s)>1,$$
where $\Lambda_{F}(n)=b(n)\log n$ is the generalized von Mangoldt function. If $N_{F} (T)$ counts the number of zeros of $F(s)\in{\mathcal S}$ in the rectangle   $0\leq Re(s)\leq1$, $0<Im(s)\leq T$ (according to multiplicities) one can show by standard contour integration the formula
\begin{equation}\label{eq.1}N_{F}(T)=\frac{d_{F}}{2\pi}T\log T+c_{1}T+O\left(\log T\right),\end{equation}
where $$c_{1}=\frac{1}{2\pi}(\log q_{F}-d_{F}(\log(2\pi)+1)$$
and
$$ q_{F}=\frac{(2\pi)^{d_{F}}Q_{F}^{2}}{\prod_{j=1}^{r}\lambda_{j}^{-2\lambda_{j}}}$$
in analogy to the Riemann-von Mangoldt formula for Riemann's zeta-function $\zeta(s)$, the prototype of
an element in ${\mathcal S}$. For more details concerning the Selberg class we refer to the surveys of Kaczorowski \cite{kac} and Perelli \cite{pe}.\\

\section{Review on the Li criterion for the Selberg class}\label{sec.2}
Let $F$ be a function in the Selberg class non-vanishing at $s=1$ and let us define the xi-function $\xi_{F}(s)$ by $$\xi_{F}(s)=s^{m_{F}}(s-1)^{m_{F}}\phi_{F}(s).$$
The function $\xi_{F}(s)$ satisfies the functional equation $$\xi_{F}(s)=\omega\overline{\xi_{F}(1-\overline{s})}.$$
The function $\xi_{F}$ is an entire function of order 1. Therefore,  the Hadamard factorization theorem implies that the function $\xi_{F}(s)$ possesses a representation as the product over its zeros
$$\xi_{F}(s)=\xi_{F}(0)e^{b_{F}s}\prod_{\rho}\left(1-\frac{s}{\rho}\right)e^{\frac{s}{\rho}},$$
where $b_{F}=\frac{\xi'_{F}}{\xi_{F}}(0)$. Inserting the following test function (firstly considered by K. Barner \cite{Bar})
$$ G_{s}(x)=\left\{\begin{array}{crll}0,&x>0,\\1/2,&x=0,\\e^{(s-1/2)x},&x<0\end{array}\right.\ \ s\in{\cb},\ \ Re(s)>1$$
into the Weil explicit formula \cite[Proposition]{OM2}, we obtain for $F(s)\in{{\mathcal S}}$ non-vanishing at $s=0$ and for all $s\in{\cb}$ different from zeros of $\xi_{F}(s)$
\begin{equation}\label{eq.2}\frac{\xi'_{F}}{\xi_{F}}(s)=\lim_{T\rightarrow\infty}\sum_{|Im(\rho)|\leq T}\frac{ord\rho}{s-\rho}.\end{equation}
Hence, the function $\xi_{F}(s)$ can be written as \begin{equation}\label{eq.3}\xi_{F}(s)=\xi_{F}(0)\prod_{\rho}\left(1-\frac{s}{\rho}\right),\end{equation}
where the product is over all zeros of $\xi_{F}(s)$ in the order given by  $|Im(\rho)|<T$ for $T \rightarrow\infty$.\\

 Let $\lambda_{F}(n)$, $n\in \mathbb{Z}$, be a sequence of numbers defined by a sum over the non-trivial zeros of $F(s)$ as  $$\lambda_{F}(n)=\sum_{\rho}\left[1-\left(1-\frac{1}{\rho}\right)^{n}\right]$$
where the sum over $\rho$ is $$\sum_{\rho}=\lim_{T\mapsto\infty}\sum_{|Im \rho |\leq T }.$$
These coefficients are well defined, indeed: the function $\xi_{F}$ is an entire function of order one, hence the series $\sum_{\rho}\frac{1}{\rho^{k}}$ converges absolutely for every integer $k\geq2$. From (\ref{eq.2}) and (\ref{eq.3}), one has
$$\sum_{\rho}\frac{1}{\rho}=\lim_{T\mapsto\infty}\sum_{|Im \rho |\leq T }\frac{1}{\rho},$$
then  the series
$$\lambda_{F}(n)=\sum_{k=1}^{n}(-1)^{k-1}\left(_k^{n}\right)\sum_{\rho}\frac{1}{\rho^{k}}$$
exists for every positive integer $n$. Let $\mathcal{Z}$ the multi-set of zeros of $\xi_{F}(s)$ (counted with multiplicity). The multi-set $\mathcal{Z}$ is invariant under the map $\rho \mapsto 1-\overline{\rho}$. We have
$$1-\left(1-\frac{1}{\rho} \right)^{-n}=1-\left(\frac{\rho-1}{\rho} \right)^{-n}=1-\left(\frac{-\rho}{1-\rho}\right)^{-n}=1-\overline{\left(1-\frac{1}{1-\overline{\rho}} \right)^{n}}$$and this gives the symmetry $\lambda_{F}(-n)=\overline{\lambda_{F}(n)}.$ This proves that the $\lambda_{F}(n)$ exists for all integer $n$. Since $\xi_{F}(s)$ is an entire function of order 1, and its zeros lie in the critical strip $0\leq Re(s)\leq1$, we also obtain that the series
$$\sum_{\rho}\frac{1+|Re(\rho)|}{(1+|\rho|)^{2}}$$ is convergent. Now the application of the Lagarias lemma \cite[Lamma 1]{lagarias} to the multi-set ${\mathcal Z}$ of non-trivial zeros of $F(s)$, the series
$$Re\lambda_{F}(n)=\sum_{\rho}Re\left[1-\left(1-\frac{1}{\rho}\right)^{n}\right]$$
converges absolutely for all integer $n$. These coefficients are expressible in terms of power-series coefficients of functions constructed from the $\xi_{F}$-function. For $n\leq -1,$ the Li coefficients $\lambda_{F}(n)$ correspond to the following Taylor expansion at the point $s=1$
$$\frac{d}{dz}\log \xi_{F}\left(\frac{1}{1-z} \right)=\sum_{n=0}^{+\infty} \lambda_{F}(-n-1)z^{n}$$
and for $n \geq1,$ they correspond to the Taylor expression at $s=0$
$$\frac{d}{dz}\log \xi_{F}\left(\frac{-z}{1-z} \right)=\sum_{n=0}^{+\infty} \lambda_{F}(n+1)z^{n}.$$
Using   \cite[Corollary 1]{B-L}, we get the following generalization of the Li criterion for the Riemann hypothesis.
\begin{theorem}\label{th.1} Let $F(s)$ be a function in the Selberg class $\mathcal{S}$ non-vanishing at $s=1.$ Then, all non-trivial zeros of $F(s)$ lie on the line $Re(s)=1/2$ if and only if $Re(\lambda_{F}(n))>0$ for $n=1,2,...$
\end{theorem}
Under the same hypothesis of Theorem \ref{th.1}, the Riemann hypothesis is also equivalent to each of the two following conditions (a) or (b):\\
\textbf{(a)} For each $\epsilon >0,$ there is a positive constant $c(\epsilon)$ such that $$Re(\lambda_{F}(n)) \geq -c(\epsilon) e^{c\epsilon} \ \ for\  all \ n \geq 1$$
\textbf{(b)} The Li coefficients $\lambda_{F}(n)$ satisfy $$\lim_{n\mapsto\infty}|\lambda_{F}(n)|^{1/n}\leq 1.$$
The proof is the same as in \cite[Theorem 2.2]{lagarias}. Next, we recall the following explicit formula for the coefficients $\lambda_{F}(n).$ Let consider the following hypothesis:
\textbf{$\mathcal{H}$ there exists a constant $c>0$ such that $F(s)$ is non-vanishing in the region:}
$$\left\{s=\sigma+it; \ \sigma \geq 1-\frac{c}{\log (Q_{F}+1+|t|)} \right\}.$$
\begin{theorem}\label{th.2} Let $F(s)$ be a function in the Selberg class $\mathcal{S}$ satisfying $\mathcal{H}.$ Then, we have
\begin{eqnarray}\label{eq.4}\lambda_{F}(-n) & = &m_{F}+n\left( \log Q_{F}-\frac{d_{F}}{2}\gamma \right) \nonumber\\
& - &\sum_{l=1}^{n} (_{l}^{l}) \frac{(-1)^{l-1}}{(l-1)!} \lim_{X\mapsto+\infty} \left\{ \sum_{k\leq X} \frac{\Lambda_{F}(k)}{k} (\log k)^{l-1}-\frac{m_{F}}{l}(\log X)^{l} \right\} \nonumber\\
& + &n\sum_{j=1}^{r}\lambda_{j} \left( -\frac{1}{\lambda_{j}+\mu_{j}} + \ \sum_{l=1}^{+\infty} \frac{\lambda_{j}+\mu_{j}}{l(l+\lambda_{j}+\mu_{j})}\right) \nonumber\\
& + &\sum_{j=1}^{r}\sum_{k=2}^{n} (_{k}^{n})(-\lambda_{j})^{k} \sum_{l=0}^{+\infty} \left( \frac{1}{l+\lambda_{j}+\mu_{j}}\right)^{k},
\end{eqnarray}
where $\gamma$ is the Euler constant.
\end{theorem}
\noindent {\bf Remark.}  The class of functions from $\mathcal{S}$ satisfying $\mathcal{H}$ is a sub-class of the class considered by Smajlovic \cite{smajl} which will be noted by $\widetilde{{\mathcal S}}$. Indeed, first recall that the Prime Number Theorem for $F(s)$ concerns the asymptotic behavior of the counting function
$$\psi_{F}(x)=\sum_{n\leq x}\Lambda(n)b_{F}(n)=\sum_{n\leq x}\Lambda_{F}(n).$$
It is expected that $\psi_{F}(x)=m_{F}x+o(x)$. In \cite{K-P} Kaczorowski and Perelli shows that the Prime Number Theorem is equivalent to the non-vanishing on the 1-line. The proof is based on a weak zero-density estimate near the 1-line and on a simple almost periodicity argument. Furthermore, if $F$ a function non-vanishing at $Re(s)=1$, then a direct application of the Perron formula \cite{titch} to Dirichlet series $-\frac{F'}{F}$ implies that for any $x > 1$ (not an integer) and $T > 1$ one has
$$\psi_{F}(x)=m_{F}x-\sum_{|Im(\rho)|\leq T}\frac{x^{\rho}}{\rho}+O\left(\frac{x^{1+\epsilon}}{T}\right),$$
for some small $\epsilon> 0$. We let $T\longrightarrow\infty$ and multiply the above formula with $\frac{\log^{l} x}{x}$. Smajlovic \cite{smajl} proved   the following equivalence.
\begin{equation}\label{eq.5}\forall\  l\in{\nb},\ \lim_{x\rightarrow\infty}\log^{l}x\left(m_{F}-\frac{\psi_{F}(x)}{x}\right)=0\ \Longleftrightarrow\ \lim_{x\rightarrow\infty}\log^{l}x\sum_{\rho}\frac{x^{1-\rho}}{\rho}=0.
\end{equation}
Denote $\widetilde{{\mathcal S}}$ the set of all functions $F\in{{\mathcal S}}$, non-vanishing on the line $Re(s) = 1$
and such that (\ref{eq.5}) holds true for all positive integers $l$. Obviously $\widetilde{{\mathcal S}}\subset{\mathcal S}$. From (\ref{eq.5}), we have
$$F\in{\widetilde{{\mathcal S}}} \ \Longleftrightarrow\ F\in{{\mathcal S}}\ \hbox{and}\ \psi_{F}(x)=m_{F}x+o\left(\frac{x}{\log^{l}x}\right)$$
and if a function $F\in{{\mathcal S}}$ has a Landau type zero free region (similar to ${\mathcal H}$ above), then $F\in{\widetilde{{\mathcal S}}}$. To prove the second assertion, note that the Landau type zero free region implies, by standard analytic arguments that the error term in the prime number theorem is $O\left(\exp(-c\sqrt{log x})\right)$, hence (\ref{eq.5}) holds true. {\bf Finally, let us note here that  Theorem \ref{th.2} not hold true only for functions $F\in{{\mathcal S}}$ having a Landau type zero free region but valid for much larger class $\widetilde{{\mathcal S}}$.} Therefore, in Section \ref{sec.3}, we will consider Smajlovic class $\widetilde{{\mathcal S}}$ of $L$-functions for our study on the modified Li criterion and  the modified Li coefficients. \\

An asymptotic formula of the number $\lambda_{F}(n)$ was proved in \cite{OM4} inspired from Lagarias method \cite{lagarias} yields to a sharper error term $O(\sqrt{n}\log n)$.  To do so, we use the arithmetic formula (\ref{eq.4}). Furthermore, we prove that is equivalent to the Riemann hypothesis.
\begin{theorem}\label{th.3} Let $F \in \mathcal{S}.$ Then
$$RH\ \Leftrightarrow\ \lambda_{F}(n)=\frac{d_{F}}{2}n \log n + c_{F}n+ O(\sqrt{n}\log n),$$
where
$$c_{F}=\frac{d_{F}}{2}(\gamma-1)+\frac{1}{2}\log(\lambda Q_{F}^{2}), \ \ \ \lambda=\prod_{j=1}^{r}\lambda_{j}^{2\lambda_{j}}$$
and $\gamma$ is the Euler constant.
\end{theorem}
\noindent  Theorem \ref{th.3} was proved also in \cite{maz} using the saddle-point method in conjunction
with the theory of the N$\ddot{o}$rlund-Rice integrals.
\section{Reformulation of the Li criterion and Modified Li's coefficients}\label{sec.3}
Let $F\in{\widetilde{{\mathcal S}}}$. Define the "modified Li-coefficients" as follows
$$\lambda_{F}(n,a)=\sum_{\rho\in{Z(F)}}\left[1-\left(\frac{\rho-a}{\rho+a-1}\right)^{n}\right],$$
where the sum over $\rho$ is $$\sum_{\rho}=\lim_{T\mapsto\infty}\sum_{|Im \rho |\leq T }.$$ The last sum is *-convergent for all $n\in{\nb}$ and
$$Re(\lambda_{F}(n,a)=\sum_{\rho\in{Z(F)}}Re\left[1-\left(\frac{\rho-a}{\rho+a-1}\right)^{n}\right]$$
converges absolutely for all $n\in{\nb}$. The multi-set $\mathcal{Z}$ of zeros of $\xi_{F}(s)$ is invariant under the map $\rho \mapsto 1-\overline{\rho}$ implies that  $\lambda_{F}(-n,a)=\overline{\lambda_{F}(n,a)}$. Then $Re\lambda_{F}(-n,a)=Re\lambda_{F}(n,a)$ for all $n\in{\nb}$. To obtain a new reformulation of the Li criterion we need to modify Bombieri-Lagarias Theorem \cite[Theorem 1]{B-L}.
\begin{lemma}\label{lem.mod}{\bf (Modified Bombieri-Lagarias Theorem)}\footnote{If we take $\beta<a$, we just change 2) and 3) as follows :\\
2)	$\sum_{\rho\in{R}}Re\left[1-\left(\frac{\rho-a}{\rho+a-1}\right)^{n}\right]\leq0,$ for all $n=1,2,...$\\
3) For every fixed $\epsilon>0$, there is a constant $c(\epsilon)>0$ such that
	$$\sum_{\rho\in{R}}Re\left[1-\left(\frac{\rho-a}{\rho+a-1}\right)^{n}\right]\leq c(\epsilon)e^{\epsilon\ n},\ \hbox{for all}\ n=1,2,...$$}. Let $\beta$ and $a$ be a real numbers such that $\beta>a$ and $R$ a multiset of complex numbers $\rho$ such that :\\
i) $a-2\beta\notin{R}$,\\
ii) $$\sum_{\rho\in{R}}\frac{1+|Re(\rho)|}{(1+|\rho+a-2\beta|)^{2}}<+\infty.$$

\noindent Then, the following conditions are equivalent
\begin{enumerate}
	\item $Re(\rho)<\beta$\ for all $\rho$.
	\item $\sum_{\rho\in{R}}Re\left[1-\left(\frac{\rho-a}{\rho+a-1}\right)^{n}\right]\geq0,$ for all $n=1,2,...$
	\item For every fixed $\epsilon>0$, there is a constant $c(\epsilon)>0$ such that
	$$\sum_{\rho\in{R}}Re\left[1-\left(\frac{\rho-a}{\rho+a-1}\right)^{n}\right]\geq-c(\epsilon)e^{\epsilon\ n},\ \hbox{for all}\ n=1,2,...$$
\end{enumerate}
\end{lemma}
\noindent{\bf Remark.} If $\rho$ and $\overline{\rho}$ are in $R$, we don't need to take the real part since the expression is real.
\begin{proof} For the proof it suffices to observe that for $\rho=\beta+i\gamma$, one has
$$\left|\frac{\rho-a}{\rho+a-2\beta}\right|=\left|\frac{\beta-a+i\gamma}{a-\beta+i\gamma}\right|=1$$
and for $\rho=\beta'+i\gamma$
$$\left|\frac{\rho-a}{\rho+a-2\beta}\right|^{2}=1+\frac{4(\beta-a)(\beta'-\beta)}{|\rho+a-2\beta|^{2}}.$$
Then, Lemma \ref{lem.mod} follows by adapting Bombieri and Lagarias proof \cite[Theorem 1]{B-L}.
\end{proof}
 Using Lemma \ref{lem.mod}, we deduce a new reformulation of the Li criterion (or the modified Li's criterion).
\begin{theorem}\label{th.mod}\footnote{If we take $\beta<a$, we just change 2) and 3) as follows :\\
2)	$\sum_{\rho\in{R}}Re\left[1-\left(\frac{\rho-a}{\rho+a-1}\right)^{n}\right]\leq0,$ for all $n=1,2,...$\\
3) For every fixed $\epsilon>0$, there is a constant $c(\epsilon)>0$ such that
	$$\sum_{\rho\in{R}}Re\left[1-\left(\frac{\rho-a}{\rho+a-1}\right)^{n}\right]\leq c(\epsilon)e^{\epsilon\ n},\ \hbox{for all}\ n=1,2,...$$}
Let $\beta$ and $a$ be a real numbers such that $\beta>a$ and $R$ a multiset of complex numbers $\rho$ such that :\\
i) $a-2\beta\notin{R}$,\ $-a\notin{R}$,\\
ii) $$\sum_{\rho\in{R}}\frac{1+|Re(\rho)|}{(1+|\rho+a-2\beta|)^{2}}<+\infty$$
and
$$\sum_{\rho\in{R}}\frac{1+|Re(\rho)|}{(1+|\rho-a)^{2}}<+\infty.$$
iii) If $\rho\in{R}$ then $2\beta-\rho\in{R}$ with the same multiplicity as $\rho$.

\noindent Then, the following conditions are equivalent
\begin{enumerate}
	\item $Re(\rho)=\beta$\ for all $\rho$.
	\item $\sum_{\rho\in{R}}Re\left[1-\left(\frac{\rho-a}{\rho+a-1}\right)^{n}\right]\geq0,$ for all $n=1,2,...$
	\item For every fixed $\epsilon>0$, there is a constant $c(\epsilon)>0$ such that
	$$\sum_{\rho\in{R}}Re\left[1-\left(\frac{\rho-a}{\rho+a-1}\right)^{n}\right]\geq-c(\epsilon)e^{\epsilon\ n},\ \hbox{for all}\ n=1,2,...$$
\end{enumerate}
\end{theorem}
\begin{proof}
The proof is the same as in \cite[Theorem 1]{B-L}, so we omit it.
\end{proof}
Now, we are ready to state a reformulation of the  Li criterion for the Selberg class.
\begin{theorem}\label{th.mod2}\footnote{If we assume $a>1/2$, then the modified Li criterion is written as follows :\\
All non-trivial zeros of $F$ lie on the line $Re(s)=1/2$ if and only if $Re(\lambda_{F}(n,a))\leq0$ for all $n\in{\nb}$.}
Let $a$ be a real number such that $a<1/2$ and $F\in{\widetilde{{\mathcal S}}}$ be a function such that $a\notin{Z(F)}$. Then, all non-trivial zeros of $F$ lie on the line $Re(s)=1/2$ if and only if $Re(\lambda_{F}(n,a))\geq0$ for all $n\in{\nb}$.
\end{theorem}
\begin{proof}
Theorem \ref{th.mod} with $\beta=1/2$ and $R$ is the multiset $Z(F)$ yields to $Re(\rho)=1/2$ if and only if $Re(\lambda_{F}(n,a))\geq0$ for all $n\in{\nb}$.
\end{proof}
\noindent{\bf Remarks.}
\begin{itemize}
	\item From Theorem \ref{th.mod2} and the footnote 3, we don't need to call this criterion as in the origin definition of Li \cite{Li} for the classical zeta function "positivity Li criterion". \\
\item Theorem \ref{th.mod2} can be proved without the use of Theorem \ref{th.mod}. Indeed, let $a<1/2$ and $\rho=\beta+i\gamma$ a non-trivial zero of $F$. Observe that
$$\left|\frac{\rho-a}{\rho+a-1}\right|^{2}=1+\frac{(1-2a)(2\beta-1)}{|\rho+a-1|^{2}}.$$
Therefore, there exist at least $\rho$ such that $\left|\frac{\rho-a}{\rho+a-1}\right|>1$. Since $\frac{(1-2a)(2\beta-1)}{|\rho+a-1|^{2}}$ tends to 0 if $|\rho|$ tends to $\infty$. Then, the maximum over $\rho$ is achieved and the only finitely zeros $\rho_{k}$ such that $|\frac{\rho-a}{\rho+a-1}|=1+t=k=\max$. For the remainder zeros $|\frac{\rho-a}{\rho+a-1}|\leq1+t-\delta$ for some $\delta>0$. Hence
$$1-\left(\frac{\rho_{k}-a}{\rho_{k}+a-1}\right)^{n}=1-(1+t)^{n}e^{in\theta_{k}},$$
for $n$ large. Then, using Dirichlet's theorem, we obtain
{\small\begin{eqnarray} \sum_{\rho}1-\left(\frac{\rho-a}{\rho+a-1}\right)^{n}&=&\sum_{\rho_{k}}1-\left(\frac{\rho_{k}-a}{\rho_{k}+a-1}\right)^{n}+\sum_{\rho\neq\rho_{k}}1-\left(\frac{\rho_{k}-a}{\rho_{k}+a-1}\right)^{n}\nonumber\\
&=&K(1-(1+t)^{n})+O(n^{2}(1+t-\delta)^{n}).\nonumber
\end{eqnarray}}
\item  Using that $$\xi_{F}(s)=\xi_{F}(0)\prod_{\rho\in{Z(F)}}\left(1-\frac{s}{\rho}\right),$$
we can deduce easily that\footnote{Another method to prove equation \eqref{eq.lambda} is to consider the function
$$h(z)=-\frac{n(2a-1)(s-a)^{n-1}}{(s+a-1)^{n+1}}+\frac{n(2a-1)}{(s+a-1)^{2}}$$
and applay the Littlewood Theorem to the integral
$$\int_{C}h(s)\log\xi_{F}(s)ds$$
with $C$ is a rectangular contour with vertices at $\pm T\pm iT$ with real $T$ tends to infty and not coincid with any zero of $\xi_{F}(s)$.}
\begin{eqnarray}\label{eq.lambda}
\lambda_{F}(n,a)&=&\sum_{\rho\in{Z(F)}}\left[1-\left(\frac{\rho-a}{\rho+a-1}\right)^{n}\right]\nonumber\\
&=&\frac{1}{(n-1)!}\frac{d^{n}}{ds^{n}}\left[(s-a)^{n-1}\log\xi_{F}(s)\right]_{s=1-a}.
\end{eqnarray}
Then, we have
\begin{equation}\label{eq.serie}
\frac{d}{ds}\log\xi_{F}\left(\frac{s-a}{s+a-1}\right)=\sum_{n=0}^{\infty}\lambda_{F}(-(n+1),a)(s-a)^{n}.
\end{equation}
As a consequence, we can prove in another way the modified Li criterion on the class $\widetilde{{\mathcal S}}$ by the same argument as in \cite[Theorem 1]{brown}
\end{itemize}
\section{Arithmetic formula for the modified Li coefficients $\lambda_{F}(n,a)$}
In this section, we  obtain an arithmetic formula for $\lambda_{F}(n,a)$. The proof is based on the use of the Weil explicit formula written in the context of the Selberg class with a suitable test function.\\

First, let recall the Weil explicit formula.
\begin{proposition}\cite{OM1,OM2}\label{prop.weil}
 Let $f:\rb\rightarrow\cb$ satisfy the following conditions:
 \begin{itemize}
	\item $f$ is normalized,
	$$2f(x)=f(x+0)+f(x-0),\ \  \ x\in{\rb}.$$
	\item There is a number  $b>0$ such that
	$$V_{\rb}\left(f(x)e^{\left(\frac{1}{2}+b\right)|x|}\right)<\infty,$$
	where $V_{\rb}(.)$ denotes the total variation on $\rb$.
	\item For all $1\leq j\leq r$, the function $G_{j}(x)=f(x)e^{-ix\frac{Im(\mu_{j})}{\lambda_{j}}}$ satisfies
	$$G_{j}(x)+G_{j}(-x)=2f(0)+O(|x|^{\epsilon}), \ \  \epsilon>0.$$	
	\end{itemize}
	 Let $F(s)\in{{\mathcal S}}$. Then,
	{\small \begin{eqnarray}\label{eq.explicit}
	 &&\sum_{\rho}H(\rho)=m_{F}\left(H(0)+H(1)\right)+2f(0)\log Q_{F}\nonumber\\
	 	 &&\ +\ \sum_{j=1}^{r}\int_{0}^{+\infty}\left\{\frac{2\lambda_{j}G_{j}(0)}{x}-	 \frac{e^{\left[\left(1-\frac{\lambda_{j}}{2}
-Re(\mu_{j})\right)\frac{x}{\lambda_{j}}\right]}}{1-e^{-\frac{x}{\lambda_{j}}}}\left(G_{j}(x)+G_{j}(-x)\right)\right\}e^{-\frac{x}{\lambda_{j}}}dx\nonumber\\
&&\ -\ \sum_{n=1}^{\infty}\left[\frac{\Lambda_{F}(n)}{\sqrt{n}}f(\log n)+\frac{\overline{\Lambda_{F}(n)}}{\sqrt{n}}f(-\log n)\right],
	 \end{eqnarray}}
	 where $$H(s)=\int_{-\infty}^{+\infty}f(x)e^{(s-1/2)x}dx \ \ \hbox{et}\ \ \sum_{\rho}H(\rho)=\lim_{T\rightarrow\infty}\sum_{|Im(\rho)|<T}H(\rho).$$
 \end{proposition}
For more details about the proof of the above proposition, see for example the
paper \cite{Bar} of Barner. Using \cite[Proposition page.146]{Bar} the integral in the sum in the
third term in the right-hand side of (\ref{eq.3}) can be written as follows
$$\frac{1}{\pi}\int_{-\infty}^{+\infty}\frac{\widehat{G_{j}}(x)+\widehat{G_{j}}(-x)}{2}\psi\left(\frac{\lambda_{j}}{2}+Re(\mu_{j})+i\lambda_{j}x\right)dx,$$
where
$$\widehat{G_{j}}(x)=\int_{-\infty}^{+\infty}G_{j}(t)e^{itx}dt$$
is the Fourier transform of $G_{j}$ and $\psi(s)=\frac{\Gamma'}{\Gamma}(s)$.\\

Similar to \cite[Lemma 2]{B-L}, we have the following result.
\begin{lemma}
Let $a$ be a real number. For $n=1,2,...$ let consider the function
$$f_{n}(x)=\left\{\begin{array}{crll}P_{n}(x)&\hbox{if}&-\infty<x<0,\\
\frac{n}{2}(1-2a)&\hbox{if}&x=0,\\
0&\hbox{if}&0<x,
\end{array}\right.$$
where $$P_{n}(x)=e^{(1/2-a)x}\sum_{l=1}^{n}\left(_{l}^{n}\right)\frac{(1-2a)^{l}}{(l-1)!}x^{l-1}.$$
Then
$$H_{n}(s)=\int_{-\infty}^{+\infty}f_{n}(x)e^{(s-1/2)x}dx=1-\left(1+\frac{2a-1}{s-a}\right)^{n}.$$
\end{lemma}
Note that, for $a=0$, we have  $\lambda_{F}(-n,a)=\lambda_{F}(-n)$ which its arithmetic formula were stated in Theorem \ref{th.2}.
\begin{theorem}\label{th.arith}
Let $F\in{\widetilde{{\mathcal S}}}$. For $a<0$, we have \footnote{When $a=0$, we get $\lambda_{F}(-n,0)=\lambda_{F}(-n)$ as given in \cite[Theorem 2]{OM2} or \cite{smajl}.\\
If we  replace $-a$ by $b$, we can assumed $b> 0$ and this finds Sekatskii's arithmetic formula of the modified coefficients $\lambda_{\zeta}(n,b)$ (see. \cite{sek1}).}
\begin{eqnarray}\label{eq.arith2}
&\lambda_{F}(-n,a)&\nonumber\\&=&m_{F}\left[2-\left(1-\frac{2a-1}{a}\right)^{n}-\left(1+\frac{2a-1}{1-a}\right)^{n}\right]+n(1-2a)\left[\log Q_{F}-\frac{d_{F}}{2}\gamma\right]\nonumber\\
&&\ -\ \sum_{l=1}^{n}\left(_{l}^{n}\right)\frac{(2a-1)^{l}}{(l-1)!}\lim_{X\rightarrow+\infty}\left\{\sum_{k\leq X}\frac{\Lambda_{F}(k)}{k^{1-a}}\log^{l-1}k-\frac{m_{F}(l-1)!}{X^{-a}}\sum_{k=0}^{l-1}\frac{\log^{k}X}{k!(-a)^{l-k}}\right\}\nonumber\\
&&\ +\ n(1-2a)\sum_{j=1}^{r}\lambda_{j}\left(-\frac{1}{\lambda_{j}+\mu_{j}}+\sum_{l=1}^{\infty}\frac{\lambda_{j}+\mu_{j}}{l(l+\lambda_{j}+\mu_{j})}\right)\nonumber\\
&&\ -\ n(1-2a)\sum_{j=1}^{r}\lambda_{j}\sum_{k=0}^{+\infty}(a\lambda_{j})^{k}\sum_{m=0}^{+\infty}\frac{1}{(m+\lambda_{j}+\mu_{j})^{k+1}}\nonumber\\
&&\ +\ \sum_{j=1}^{r}\sum_{l=2}^{n}\left(_{l}^{n}\right)(-(1-2a)\lambda_{j})^{l}\sum_{m=0}^{+\infty}\frac{1}{(m+(1-a)\lambda_{j}+\mu_{j})^{l}}.\nonumber\\
\end{eqnarray}
\end{theorem}
\begin{proof}
For large $X$  not integer, let
$$f_{n,X}(x)=\left\{\begin{array}{crll}f_{n}(x)&\hbox{if}&-\log X<x<0,\\
\frac{1}{2}f_{n}(-\log X)&\hbox{if}&x=-\log X,\\
\frac{n}{2}(1-2a)&\hbox{if}&x=0,\\
0&&\hbox{otherwise}.
\end{array}\right.$$
Then, the function $f_{n,X}(x)$ satisfies condition of Proposition \ref{prop.weil}. Let
$$H_{n,X}(s)=\int_{-\infty}^{+\infty}f_{n,X}(x)e^{(s-1/2)x}dx.$$
Therefore, we get
$$W_{\lambda_{j},\mu_{j}}=\sum_{j=1}^{r}\int_{0}^{+\infty}\left\{\frac{2\lambda_{j}G_{j,n,X}(0)}{x}-	 \frac{e^{\left[\left(1-\frac{\lambda_{j}}{2}
-Re(\mu_{j})\right)\frac{x}{\lambda_{j}}\right]}}{1-e^{-\frac{x}{\lambda_{j}}}}\left(G_{j,n,X}(x)+G_{j,n,X}(-x)\right)\right\}e^{-\frac{x}{\lambda_{j}}}dx,$$
with
$$G_{j,n,X}(x)=f_{n,X}(x)e^{-ix\frac{Im(\mu_{j})}{\lambda_{j}}}.$$
Using the relation
$$\frac{\Gamma'}{\Gamma}(z)=\int_{0}^{+\infty}\left(\frac{e^{-u}}{u}-\frac{e^{-zu}}{1-e^{-u}}\right)du,$$
we obtain
\begin{eqnarray}\label{eq.I1+I2}
&&\int_{0}^{+\infty}\left\{\frac{2\lambda_{j}G_{j,n,X}(0)}{x}-	 \frac{e^{\left[\left(1-\frac{\lambda_{j}}{2}
-Re(\mu_{j})\right)\frac{x}{\lambda_{j}}\right]}}{1-e^{-\frac{x}{\lambda_{j}}}}\left(G_{j,n,X}(x)+G_{j,n,X}(-x)\right)\right\}e^{-\frac{x}{\lambda_{j}}}dx\nonumber\\
&&\ =\ \lim_{X\rightarrow +\infty}\int_{0}^{\log X}\left\{\frac{2\lambda_{j}f_{n,X}(0)}{x}-	 \frac{e^{\left[\left(1-\frac{\lambda_{j}}{2}
-Re(\mu_{j})\right)\frac{x}{\lambda_{j}}\right]}}{1-e^{-\frac{x}{\lambda_{j}}}}f_{n,X}(-x)e^{ix\frac{Im(\mu_{j})}{\lambda_{j}}}\right\}e^{-\frac{x}{\lambda_{j}}}dx\nonumber\\
&&\ =\ \lim_{X\rightarrow +\infty}\int_{0}^{\log X}\left\{\frac{n(1-2a)\lambda_{j}}{x}-	 \frac{e^{\left[\left(1-\frac{\lambda_{j}}{2}
-\overline{\mu_{j}}\right)\frac{x}{\lambda_{j}}\right]}}{1-e^{-\frac{x}{\lambda_{j}}}}f_{n,X}(-x)\right\}e^{-\frac{x}{\lambda_{j}}}dx\nonumber\\
&&\ =\ \lim_{X\rightarrow +\infty}\lambda_{j}\int_{0}^{\log X}\left\{\frac{n(1-2a)e^{-x}}{x}-	 \frac{e^{-\left(\frac{\lambda_{j}}{2}
+\overline{\mu_{j}}\right)x}}{1-e^{-x}}f_{n,X}(-\lambda_{j}x)\right\}dx\nonumber\\
&&\ =\ n\lambda_{j}(1-2a)\frac{\Gamma'}{\Gamma}(\lambda_{j}+\overline{\mu_{j}})+\lambda_{j}\lim_{X\rightarrow +\infty}\int_{0}^{\frac{\log X}{\lambda_{j}}}\left\{\frac{n(1-2a)e^{-(\lambda_{j}+\overline{\mu_{j}})x}}{1-e^{-x}}-	 \frac{e^{-\left(\frac{\lambda_{j}}{2}
+\overline{\mu_{j}}\right)x}}{1-e^{-x}}f_{n}(-\lambda_{j}x)\right\}dx\nonumber\\
&&\ =\ n\lambda_{j}(1-2a)\frac{\Gamma'}{\Gamma}(\lambda_{j}+\overline{\mu_{j}})\nonumber\\
&&\ \ \ \ \ +\lambda_{j}\int_{0}^{+\infty}\left\{\frac{n(1-2a)e^{-(\lambda_{j}+\overline{\mu_{j}})x}}{1-e^{-x}}-	 \frac{e^{-\left(\frac{\lambda_{j}}{2}
+\overline{\mu_{j}}\right)x}}{1-e^{-x}}\left[e^{-(\frac{1}{2}-a)\lambda_{j}x}\sum_{l=1}^{n}\left(_{l}^{n}\right)\frac{(1-2a)^{l}}{(l-1)!}(-\lambda_{j}x)^{l-1}\right]\right\}dx\nonumber\\
&&\ =\ n\lambda_{j}(1-2a)\frac{\Gamma'}{\Gamma}(\lambda_{j}+\overline{\mu_{j}})\nonumber\\
&&\ \ \ \ \ +\lambda_{j}\int_{0}^{+\infty}\left\{n(1-2a)-	 e^{a\lambda_{j}x}\sum_{l=1}^{n}\left(_{l}^{n}\right)\frac{(1-2a)^{l}}{(l-1)!}(-\lambda_{j}x)^{l-1}\right\}\frac{e^{-(\lambda_{j}+\overline{\mu_{j}})x}}{1-e^{-x}}dx\nonumber\\
&&\ =\ n\lambda_{j}(1-2a)\frac{\Gamma'}{\Gamma}(\lambda_{j}+\overline{\mu_{j}})\nonumber\\
&&\ \ \ \ \ +\lambda_{j}\int_{0}^{+\infty}\left\{n(1-2a)(1-e^{a\lambda_{j}x})-e^{a\lambda_{j}x}\sum_{l=2}^{n}\left(_{l}^{n}\right)\frac{(1-2a)^{l}}{(l-1)!}(-\lambda_{j}x)^{l-1}\right\}\frac{e^{-(\lambda_{j}+\overline{\mu_{j}})x}}{1-e^{-x}}dx\nonumber\\
&&\ =\ n\lambda_{j}(1-2a)\frac{\Gamma'}{\Gamma}(\lambda_{j}+\overline{\mu_{j}})+\lambda_{j}\int_{0}^{+\infty}\left\{n(1-2a)(1-e^{a\lambda_{j}x})\frac{e^{-(\lambda_{j}+\overline{\mu_{j}})x}}{1-e^{-x}}\right\}dx\nonumber\\
&&\ \ \ \ \ +\lambda_{j}(1-2a)\int_{0}^{+\infty}\left\{-\sum_{l=2}^{n}\left(_{l}^{n}\right)\frac{\left(-(1-2a)\lambda_{j}x\right)^{l-1}}{(l-1)!}\right\}
\frac{e^{-((1-a)\lambda_{j}+\overline{\mu_{j}})x}}{1-e^{-x}}dx\nonumber\\
&&\ =\ n\lambda_{j}(1-2a)\frac{\Gamma'}{\Gamma}(\lambda_{j}+\overline{\mu_{j}})+n(1-2a)\lambda_{j}I_{1}+(1-2a)\lambda_{j}I_{2},\nonumber\\
\end{eqnarray}
where
$$I_{1}=\int_{0}^{+\infty}\left\{(1-e^{a\lambda_{j}x})\frac{e^{-(\lambda_{j}+\overline{\mu_{j}})x}}{1-e^{-x}}\right\}dx$$
and
$$I_{2}=\int_{0}^{+\infty}\left\{-\sum_{l=2}^{n}\left(_{l}^{n}\right)\frac{\left(-(1-2a)\lambda_{j}x\right)^{l-1}}{(l-1)!}\right\}
\frac{e^{-((1-a)\lambda_{j}+\overline{\mu_{j}})x}}{1-e^{-x}}dx.$$
In order that the integral $I_{1}$ and $I_{2}$ converges, one needs to assume that $Re((a-1)\lambda_{j}-\mu_{j})<0$. Since the assumption on $\lambda_{j}$ and $\mu_{j}$ posed in the second axiom in the definition of the Selberg class yield that, this is true only if $a\leq1$.\\

\noindent Let start with $I_{2}$. For $a\leq1$, we have
\begin{equation}\label{eq.I2}
I_{2}=\sum_{l=2}^{n}\left(_{l}^{n}\right)(-(1-2a)\lambda_{j})^{l-1}\sum_{m=0}^{+\infty}\frac{1}{(m+(1-a)\lambda_{j}+\overline{\mu_{j}})^{l}}.
\end{equation}
Furthermore, for $a\leq1$, one has,
\begin{eqnarray}\label{eq.I1}
I_{1}&=&\frac{\Gamma'}{\Gamma}(\lambda_{j}(1-a)+\overline{\mu_{j}})-\frac{\Gamma'}{\Gamma}(\lambda_{j}+\overline{\mu_{j}})\nonumber\\
&=&-\sum_{k=1}^{+\infty}\frac{(a\lambda_{j})^{k}}{k!}\int_{0}^{+\infty}x^{k}\frac{e^{-(\lambda_{j}+\overline{\mu_{j}})}}{1-e^{-x}}dx\nonumber\\
&=&-\sum_{k=1}^{+\infty}(a\lambda_{j})^{k}\sum_{m=0}^{+\infty}\frac{1}{(m+\lambda_{j}+\overline{\mu_{j}})^{k+1}}.
\end{eqnarray}
In the other hand, a simple computation yields
\begin{equation}\label{eq.H0}
H_{n,X}(0)=1-\left(1-\frac{2a-1}{a}\right)^{n}+\sum_{l=1}^{n}\left(_{l}^{n}\right)(2a-1)^{l}X^{a}\sum_{k=0}^{l-1}\frac{\log^{k}X}{k!(-a)^{l-k}},
\end{equation}
\begin{eqnarray}\label{eq.H1}
H_{n,X}(1)&=&1-\left(1+\frac{2a-1}{1-a}\right)^{n}+\sum_{l=1}^{n}\left(_{l}^{n}\right)(2a-1)^{l}X^{a-1}\sum_{k=0}^{l-1}\frac{\log^{k}X}{k!(1-a)^{l-k}}\nonumber\\
&=&1-\left(1+\frac{2a-1}{1-a}\right)^{n}+O\left(\frac{\log^{n}X}{X^{1-a}}\right)
\end{eqnarray}
and
\begin{eqnarray}\label{eq.arit}
&&\sum_{n=1}^{\infty}\left[\frac{\Lambda_{F}(n)}{\sqrt{n}}f_{n,X}(\log n)+\frac{\overline{\Lambda_{F}(n)}}{\sqrt{n}}f_{n,X}(-\log n)\right]\nonumber\\&& \ \ =\ \ \sum_{l=1}^{n}\left(_{l}^{n}\right)\frac{(-1)^{l-1}}{(l-1)!}(1-2a)^{l}\sum_{k\leq X}\frac{\overline{\Lambda_{F}(k)}}{k^{1-a}}\log^{l-1}k.
\end{eqnarray}
Inserting equations \eqref{eq.I2} and \eqref{eq.I1} into formula \eqref{eq.I1+I2}, using equations \eqref{eq.H0}, \eqref{eq.H1} and \eqref{eq.arit} and   the explicit formula given by equation \eqref{eq.explicit}, for $a<0$, we get
\begin{eqnarray}\label{eq.arith2}
&&\lim_{X\rightarrow+\infty}\sum_{\rho}H_{n,X}(\rho)\nonumber\\&=&m_{F}\left[2-\left(1-\frac{2a-1}{a}\right)^{n}-\left(1+\frac{2a-1}{1-a}\right)^{n}\right]+n(1-2a)\left[\log Q_{F}-\frac{d_{F}}{2}\gamma\right]\nonumber\\
&&\ -\ \sum_{l=1}^{n}\left(_{l}^{n}\right)\frac{(2a-1)^{l}}{(l-1)!}\lim_{X\rightarrow+\infty}\left\{\sum_{k\leq X}\frac{\overline{\Lambda_{F}(k)}}{k^{1-a}}\log^{l-1}k-m_{F}(l-1)!X^{a}\sum_{k=0}^{l-1}\frac{\log^{k}X}{k!(-a)^{l-k}}\right\}\nonumber\\
&&\ +\ n(1-2a)\sum_{j=1}^{r}\lambda_{j}\left(-\frac{1}{\lambda_{j}+\overline{\mu_{j}}}+\sum_{l=1}^{\infty}\frac{\lambda_{j}+\overline{\mu_{j}}}{l(l+\lambda_{j}+\overline{\mu_{j}})}\right)\nonumber\\
&&\ -\ n(1-2a)\sum_{j=1}^{r}\lambda_{j}\sum_{k=0}^{+\infty}(a\lambda_{j})^{k}\sum_{m=0}^{+\infty}\frac{1}{(m+\lambda_{j}+\overline{\mu_{j}})^{k+1}}\nonumber\\
&&\ +\ \sum_{j=1}^{r}\sum_{l=2}^{n}\left(_{l}^{n}\right)(-(1-2a)\lambda_{j})^{l}\sum_{m=0}^{+\infty}\frac{1}{(m+(1-a)\lambda_{j}+\overline{\mu_{j}})^{l}}.\nonumber\\
\end{eqnarray}
Under condition made above on $F(s)$, since  $a<0$, we prove by the same argument as in \cite[page. 56 and 57]{OM1} that
$$\lim_{X\rightarrow+\infty}\sum_{\rho}H_{n,X}(\rho)=\sum_{\rho}H(\rho)=\lambda_{F}(n,a).$$
Furthermore, recall that $\lambda_{F}(-n,a)=\overline{\lambda_{F}(n,a)}$, then, for $a<0$, we obtain

\begin{eqnarray}\label{eq.arith2}
\lambda_{F}(-n,a)&=&m_{F}\left[2-\left(1-\frac{2a-1}{a}\right)^{n}-\left(1+\frac{2a-1}{1-a}\right)^{n}\right]+n(1-2a)\left[\log Q_{F}-\frac{d_{F}}{2}\gamma\right]\nonumber\\
&&\ -\ \sum_{l=1}^{n}\left(_{l}^{n}\right)\frac{(2a-1)^{l}}{(l-1)!}\lim_{X\rightarrow+\infty}\left\{\sum_{k\leq X}\frac{\Lambda_{F}(k)}{k^{1-a}}\log^{l-1}k-m_{F}(l-1)!X^{a}\sum_{k=0}^{l-1}\frac{\log^{k}X}{k!(-a)^{l-k}}\right\}\nonumber\\
&&\ +\ n(1-2a)\sum_{j=1}^{r}\lambda_{j}\left(-\frac{1}{\lambda_{j}+\mu_{j}}+\sum_{l=1}^{\infty}\frac{\lambda_{j}+\mu_{j}}{l(l+\lambda_{j}+\mu_{j})}\right)\nonumber\\
&&\ -\ n(1-2a)\sum_{j=1}^{r}\lambda_{j}\sum_{k=0}^{+\infty}(a\lambda_{j})^{k}\sum_{m=0}^{+\infty}\frac{1}{(m+\lambda_{j}+\mu_{j})^{k+1}}\nonumber\\
&&\ +\ \sum_{j=1}^{r}\sum_{l=2}^{n}\left(_{l}^{n}\right)(-(1-2a)\lambda_{j})^{l}\sum_{m=0}^{+\infty}\frac{1}{(m+(1-a)\lambda_{j}+\mu_{j})^{l}}.\nonumber\\
\end{eqnarray}
{\bf Remark.}  The arithmetic formula given in Theorem \ref{th.arith} can be written in a simple way as follows
\begin{eqnarray}\label{eq.arith2}
&\lambda_{F}(-n,a)&\nonumber\\&=&m_{F}\left[2-\left(1-\frac{2a-1}{a}\right)^{n}-\left(1+\frac{2a-1}{1-a}\right)^{n}\right]+n(1-2a)\log Q_{F}\nonumber\\
&&\ -\ \sum_{l=1}^{n}\left(_{l}^{n}\right)\frac{(2a-1)^{l}}{(l-1)!}\lim_{X\rightarrow+\infty}\left\{\sum_{k\leq X}\frac{\Lambda_{F}(k)}{k^{1-a}}\log^{l-1}k-\frac{m_{F}(l-1)!}{X^{-a}}\sum_{k=0}^{l-1}\frac{\log^{k}X}{k!(-a)^{l-k}}\right\}\nonumber\\
&&\ +\ n(1-2a)\sum_{j=1}^{r}\psi(\lambda_{j}(1-a)+\mu_{j})\nonumber\\
&&\ +  \sum_{j=1}^{r}\sum_{l=2}^{n}\left(_{l}^{n}\right)(-(1-2a)\lambda_{j})^{l}\sum_{m=0}^{+\infty}\frac{1}{(m+(1-a)\lambda_{j}+\mu_{j})^{l}}.\nonumber
\end{eqnarray}

\end{proof}
\section{An asymptotic formula}
In this section, we use the same argument as in \cite[Theorem 4.1]{maz} which use the saddle point method in conjunction with the N$\ddot{o}$rlund-Rice integral to deduce an asymptotic formula for the modified Li coefficients equivalent to the Riemann hypothesis.
\begin{theorem}\label{th.asymp}
Let $F\in{\widetilde{{\mathcal S}}}$. For $ a<0$, we have
$$RH$$
 $$ \Longleftrightarrow$$
 \begin{eqnarray}\lambda_{F}(n,a)&=&\left(\frac{1}{2}-a\right)d_{F}n\log n\nonumber\\ &&\ +\ \left\{\frac{d_{F}}{2}\left(\gamma-1-\log(1-2a)\right) +\ \frac{1}{2}\log(\lambda Q_{F}^{2})+C_{F}(a)\right\}(1-2a)n \nonumber\\\ &&+\ O\left(\sqrt{n}\log n\right),\nonumber\end{eqnarray}
where
$$C_{F}(a)=\sum_{j=1}^{r}\sum_{k=0}^{+\infty}(a\lambda_{j})^{k}\zeta(k,\lambda_{j}+\mu_{j}),\ \ \lambda=\prod_{j=1}^{r}\lambda_{j}^{2\lambda_{j}},$$
 $\gamma$ is the Euler constant and $\zeta(s, q)$ is the Hurwitz zeta function given by
$$\zeta(s, q) =\sum_{n=0}^{+\infty}\frac{1}{(n + q)^{s}}.$$
\end{theorem}
\begin{proof}Without loss of generality, we assume that $\mu_{j}$ is a real number.
First, write the arithmetic formula of $\lambda_{F}(-n,a)$ (equation \eqref{eq.arith2}) as
\begin{eqnarray}\label{eq.arit3}
\lambda_{F}(-n,a)&=&m_{F}\left[2-\left(1-\frac{2a-1}{a}\right)^{n}-\left(1+\frac{2a-1}{1-a}\right)^{n}\right]\nonumber\\
&&\ +\ n(1-2a)\left[\log Q_{F}-\frac{d_{F}}{2}\gamma\right] +\ S_{F}(n,a)+n(1-2a)C_{F}(a)+S_{1}+S_{2},\nonumber\\
\end{eqnarray}
where
$$S_{F}(n,a)=-\ \sum_{l=1}^{n}\left(_{l}^{n}\right)\frac{(2a-1)^{l}}{(l-1)!}\lim_{X\rightarrow+\infty}\left\{\sum_{k\leq X}\frac{\Lambda_{F}(k)}{k^{1-a}}\log^{l-1}k-m_{F}(l-1)!X^{a}\sum_{k=0}^{l-1}\frac{\log^{k}X}{k!(-a)^{l-k}}\right\},$$
$$C_{F}(a)=-\sum_{j=1}^{r}\lambda_{j}\sum_{k=0}^{+\infty}(a\lambda_{j})^{k}\sum_{m=0}^{+\infty}\frac{1}{(m+\lambda_{j}+\mu_{j})^{k+1}},$$
$$S_{1}= n(1-2a)\sum_{j=1}^{r}\lambda_{j}\left(-\frac{1}{\lambda_{j}+\mu_{j}}+\sum_{l=1}^{\infty}\frac{\lambda_{j}+\mu_{j}}{l(l+\lambda_{j}+\mu_{j})}\right)$$
and
$$S_{2}= \sum_{j=1}^{r}\sum_{l=2}^{n}\left(_{l}^{n}\right)(-(1-2a)\lambda_{j})^{l}\sum_{m=0}^{+\infty}\frac{1}{(m+(1-a)\lambda_{j}+\mu_{j})^{l}}.$$
Write the sum $S_{2}$ as follows
$$S_{2}=\sum_{j=1}^{r}I_{j},$$
where
\begin{eqnarray}I_{j}&=&\sum_{l=2}^{n}\left(_{l}^{n}\right)(-(1-2a)\lambda_{j})^{l}\sum_{m=0}^{+\infty}\frac{1}{(m+(1-a)\lambda_{j}+\mu_{j})^{l}}\nonumber\\
&=&\sum_{l=2}^{n}\left(_{l}^{n}\right)(-1)^{l}\frac{\zeta(l,(1-a)\lambda_{j}+\mu_{j})}{\left[\left((1-2a)\lambda_{j}\right)^{-1}\right]^{l}},\nonumber
 \end{eqnarray}
 where $$\zeta(s, q) =\sum_{n=0}^{+\infty}\frac{1}{(n + q)^{s}}$$
 is the Hurwitz zeta function. Using the same notation  of $H_{n}(m,k)$ \cite[Equation 4.1]{maz}, we get
 $$I_{j}=H_{n}\left(\frac{1-a}{1-2a}+\frac{\mu_{j}}{(1-2a)\lambda_{j}},\frac{1}{(1-2a)\lambda_{j}}\right).$$
Now applying \cite[Proposition 4.3]{maz} with $m=\frac{1-a}{1-2a}+\frac{\mu_{j}}{(1-2a)\lambda_{j}}$ and $k=\frac{1}{(1-2a)\lambda_{j}}$, we obtain
\begin{eqnarray}
I_{j}&=&\left[(1-a)\lambda_{j}+\mu_{j}-\frac{1}{2}\right]\nonumber\\&&\ -\ n(1-2a)\lambda_{j}\left\{\psi\left((1-a)\lambda_{j}+\mu_{j}\right)-\log\left((1-2a)\lambda_{j}\right)+1-h_{n-1}\right\}\nonumber\\
&&+\ a_{n}\left(\frac{1-a}{1-2a}+\frac{\mu_{j}}{(1-2a)\lambda_{j}},\frac{1}{(1-2a)\lambda_{j}}\right),\nonumber
\end{eqnarray}
where the $a_{n}$ are exponentially small and $a_{n}=O(1)$ (see. \cite[Proposition 4.3]{maz} for an explicit expression of $a_{n}$). Here, $h_{n} = 1 +\frac{1}{2}+\frac{1}{3}+...+\frac{1}{n}$ is a harmonic number, and $\psi(x)$ is the logarithm derivative of the Gamma function. Therefore
\begin{eqnarray}\label{eq.Ij}
I_{j}&=&\left[(1-a)\lambda_{j}+\mu_{j}-\frac{1}{2}\right]\nonumber\\
&&\ -n(1-2a)\lambda_{j}\left[\psi\left((1-a)\lambda_{j}+\mu_{j}\right)-\log\left((1-2a)\lambda_{j}\right)+1-h_{n-1}\right]+O(1).\nonumber\\
\end{eqnarray}
Summing \eqref{eq.Ij} over $j$, we get
\begin{eqnarray}\label{eq.sumIj}
\sum_{j=1}^{r}I_{j}&=&\sum_{j=1}^{r}\left[(1-a)\lambda_{j}+\mu_{j}-\frac{1}{2}\right]\nonumber\\
&&\ -n(1-2a)\sum_{j=1}^{r}\lambda_{j}\left[\psi\left((1-a)\lambda_{j}+\mu_{j}\right)-\log\left((1-2a)\lambda_{j}\right)+1-h_{n-1}\right]+O(1).\nonumber\\
\end{eqnarray}
Using the expression $$\psi(z)=-\gamma-\frac{1}{z}+\sum_{l=1}^{+\infty}\frac{z}{l(l+z)},$$
where $\gamma$ is the Euler constant and the estimate
$$h_{n}=\log n-\gamma+\frac{1}{2n}+O\left(\frac{1}{2n^{2}}\right),$$
we deduce from \eqref{eq.arit3}, \eqref{eq.sumIj} and for $a<0$
$$\left[2-\left(1-\frac{2a-1}{a}\right)^{n}-\left(1+\frac{2a-1}{1-a}\right)^{n}\right]=O(1)$$
 that
\begin{eqnarray}
&&\lambda_{F}(-n,a)=\left(\sum_{j=1}^{r}\lambda_{j}\right)n(1-2a)\log n\nonumber\\
&&\ +\ n(1-2a)\left\{\left(\sum_{j=1}^{r}\lambda_{j}\right)(\gamma-1)+\log Q_{F}-\sum_{j=1}^{r}\lambda_{j}\log[(1-2a)\lambda_{j}]+C_{F}(a)\right\}\nonumber\\&&\ +\ S_{F}(n,a)+O(1).\nonumber\\
\end{eqnarray}
With the notation $d_{F}=\sum_{j=1}^{r}\lambda_{j}$ and $\lambda=\prod_{j=1}^{r}\lambda_{j}^{2\lambda_{j}}$, we obtain
\begin{eqnarray}
\lambda_{F}(-n,a)&=&(1/2-a)d_{F}n\log n\nonumber\\
&&\ +\ n(1-2a)\left\{\frac{d_{F}}{2}\left[\gamma-1-\log(1-2a)\right]+\frac{1}{2}\log (\lambda Q_{F})^{2}+C_{F}(a)\right\}\nonumber\\&&\ +\ S_{F}(n,a)+O(1),\nonumber\\
\end{eqnarray}
where
$$S_{F}(n,a)=-\ \sum_{l=1}^{n}\left(_{l}^{n}\right)\frac{(2a-1)^{l}}{(l-1)!}\lim_{X\rightarrow+\infty}\left\{\sum_{k\leq X}\frac{\Lambda_{F}(k)}{k^{1-a}}\log^{l-1}k-m_{F}(l-1)!X^{a}\sum_{k=0}^{l-1}\frac{\log^{k}X}{k!(-a)^{l-k}}\right\}.$$
A simple adaptation of the argument used in \cite[Lemma 4.4]{maz} yields to that, if the Riemann hypothesis holds for $F\in{\widetilde{{\mathcal S}}}$, then
$$S_{F}(n,a)=O\left(\sqrt{n}\log n\right).$$
Now, collecting all results above and noting that $\lambda_{F}(-n,a)=\overline{\lambda_{F}(n,a)}$, we deduce that the Riemann hypothesis implies
\begin{eqnarray}
\lambda_{F}(n,a)&=&(1/2-a)d_{F}n\log n\nonumber\\
&&\ +\ n(1-2a)\left\{\frac{d_{F}}{2}\left[\gamma-1-\log(1-2a)\right]+\frac{1}{2}\log (\lambda Q_{F})^{2}+C_{F}(a)\right\}\nonumber\\&&\ +\ O\left(\sqrt{n}\log n\right).\nonumber
\end{eqnarray}
Conversely, if
\begin{eqnarray}
\lambda_{F}(n,a)&=&(1/2-a)d_{F}n\log n\nonumber\\
&&\ +\ n(1-2a)\left\{\frac{d_{F}}{2}\left[\gamma-1-\log(1-2a)\right]+\frac{1}{2}\log (\lambda Q_{F})^{2}+C_{F}(a)\right\}\nonumber\\&&\ +\ O\left(\sqrt{n}\log n\right).\nonumber
\end{eqnarray}
then, $\lambda_{F}(n,a)$ grows polynomially in $n$. Therefore, if the Riemann hypothesis is false then from Theorem \ref{th.mod}, some modified Li
coefficients become exponentially large in $n$ and negative and the asymptotic formula of  $\lambda_{F}(n,a)$ rules
out.
\end{proof}
\section{Another proofs of Theorems \ref{th.arith} and \ref{th.asymp}}
The arithmetic formula proved in Theorem \ref{th.arith} can be obtained by another way by using that the coefficients
$$\lambda_{F}(n,a)=\sum_{\rho\in{Z(F)}}\left[1-\left(\frac{\rho-a}{\rho+a-1}\right)^{n}\right]
=\frac{1}{(n-1)!}\frac{d^{n}}{ds^{n}}\left[(s-a)^{n-1}\log\xi_{F}(s)\right]_{s=1-a}.$$
Therefore, the arithmetic formula \eqref{eq.arith2} can be deduced by the same argument used by Smajlovic in \cite[Appendix A]{smajl} or by the  author in \cite{OM3}. Next, we will separately investigate the asymptotic behavior of archimedean and non-archimedean contribution in
the arithmetic formula given in Theorem \ref{th.arith} or established in another way below. To do so, we use a recurrence relation
for the digamma function \cite[6.4.6]{A-S}, for $n\geq0$
$$\psi^{(n)}(z+1)=\psi^{(n)}(z)+(-1)^{n}n!z^{-n-1}$$
and
\begin{equation}\label{eq.psi}\psi^{(n)}(z)=(-1)^{n+1}n!\zeta(n+1,z)\end{equation}
for $z\neq0,-1,-2,...$ We have,
\begin{eqnarray}
\lambda_{F}(n,a)&=&\frac{1}{(n-1)!}\frac{d^{n}}{ds^{n}}\left[(s-a)^{n-1}\log\xi_{F}(s)\right]_{s=1-a}\nonumber\\
&=&\sum_{k=1}^{n}\left(_{k}^{n}\right)\frac{(1-2a)^{k}}{(k-1)!}\left[\frac{d^{k}}{ds^{k}}\log\xi_{F}(s)\right]_{s=1-a}\nonumber
\end{eqnarray}
Using that
$$\frac{\xi'_{F}}{\xi_{F}}(s)=\frac{m_{F}}{s}+\frac{m_{F}}{s-1}+\log Q_{F}+\frac{F'}{F}(s)+\sum_{j=1}^{r}\frac{\Gamma'}{\Gamma}(\lambda_{j}s+\mu_{j}),$$
we get
\begin{eqnarray}\label{eq.art2}
\lambda_{F}(n,a)&=&\sum_{k=1}^{n}\left(_{k}^{n}\right)\frac{(1-2a)^{k}}{(k-1)!}\left[\frac{d^{k-1}}{ds^{k-1}}\frac{F'}{F}(s)\right]_{s=1-a}+m_{F}\sum_{k=1}^{n}\left(_{k}^{n}\right)\frac{(1-2a)^{k}}{(1-a)^{k}}\nonumber\\
&&\ +\ \sum_{k=1}^{n}\left(_{k}^{n}\right)\frac{(1-2a)^{k}}{(k-1)!}\ \delta_{k,1}\ \log Q_{F}-m_{F}\sum_{k=1}^{n}\left(_{k}^{n}\right)\frac{(1-2a)^{k}}{(-a)^{k}}\nonumber\\
&&\ +\ \sum_{k=1}^{n}\left(_{k}^{n}\right)\frac{(1-2a)^{k}}{(k-1)!}\sum_{j=1}^{r}\lambda_{j}^{k}\psi^{(k-1)}(\lambda_{j}(1-a)+\mu_{j})\nonumber\\
&=& m_{F}\left[2-\left(1-\frac{2a-1}{a}\right)^{n}-\left(1+\frac{2a-1}{1-a}\right)^{n}\right]+ n(1-2a)\log Q_{F}\nonumber\\
&&\ +\ \sum_{k=1}^{n}\left(_{k}^{n}\right)(1-2a)^{k}\eta_{k-1}(F,1-a)\nonumber\\
&&\ +\ \sum_{k=1}^{n}\left(_{k}^{n}\right)\frac{(1-2a)^{k}}{(k-1)!}\sum_{j=1}^{r}\lambda_{j}^{k}\psi^{(k-1)}(\lambda_{j}(1-a)+\mu_{j}),
\end{eqnarray}
where $\frac{F'}{F}(s)=\sum_{l=-1}^{+\infty}\eta_{k-1}(s-(1-a))^{l}$ is the Laurent expansion at $s=1-a$ and $\delta_{k,1}=1$ if $k=1$ and 0 otherwise. In particular, $\eta_{-1}=-m_{F}$ if $F$ has a pole at $s=1-a$ and $\eta_{-1}=0$ otherwise. It's easy to see that, for $a<0$
\begin{eqnarray}\label{eq.eta}
\eta_{k-1}&=&\frac{1}{(k-1)!}\left[\frac{d^{k-1}}{ds^{k-1}}\left(\frac{m_{F}}{s-1+a}+\frac{F'}{F}(s)\right)\right]_{1-a}\nonumber\\
&=&\ -\ \frac{1}{(k-1)!}\lim_{X\rightarrow+\infty}\left\{\sum_{k\leq X}\frac{\Lambda_{F}(k)}{k^{1-a}}\log^{l-1}k-m_{F}(l-1)!X^{a}\sum_{k=0}^{l-1}\frac{\log^{k}X}{k!(-a)^{l-k}}\right\}.\nonumber\\
\end{eqnarray}
Furthermore, the last sum in equation \eqref{eq.art2} can be written as follows
\begin{eqnarray}\label{eq.26}
&&\sum_{k=1}^{n}\left(_{k}^{n}\right)\frac{(1-2a)^{k}}{(k-1)!}\sum_{j=1}^{r}\lambda_{j}^{k}\psi^{(k-1)}(\lambda_{j}(1-a)+\mu_{j})\nonumber\\
&=&(1-2a)n\sum_{j=1}^{r}\lambda_{j}\psi(\lambda_{j}(1-a)+\mu_{j})\nonumber\\
&&\ +\ \sum_{k=2}^{n}\left(_{k}^{n}\right)(1-2a)^{k} \sum_{j=1}^{r}(-\lambda_{j})^{k}\sum_{m=0}^{+\infty}\frac{1}{(\lambda_{j}(1-a)+\mu_{j}+m)^{k}}.
\end{eqnarray}
Then, from equations \eqref{eq.art2}, \eqref{eq.eta}, \eqref{eq.26}, we find the arithmetic formula given in Theorem \ref{th.arith}.\\

Now, to prove the asymptotic formula given in Theorem \ref{th.asymp}, we write equation \ref{eq.art2}
$$\lambda_{F}(n,a)=S_{Arch}+S_{NArch},$$
where
$$S_{Arch}= n(1-2a)\log Q_{F}+\sum_{k=1}^{n}\left(_{k}^{n}\right)\frac{(1-2a)^{k}}{(k-1)!}\sum_{j=1}^{r}\lambda_{j}^{k}\psi^{(k-1)}(\lambda_{j}(1-a)+\mu_{j})$$
and
$$S_{NArch}=m_{F}\left[2-\left(1-\frac{2a-1}{a}\right)^{n}-\left(1+\frac{2a-1}{1-a}\right)^{n}\right] + \sum_{k=1}^{n}\left(_{k}^{n}\right)(1-2a)^{k}\eta_{k-1}(F,1-a).$$
Formula \eqref{eq.psi} yields to
\begin{eqnarray}
&&\sum_{k=1}^{n}\left(_{k}^{n}\right)\frac{(1-2a)^{k}}{(k-1)!}\sum_{j=1}^{r}\lambda_{j}^{k}\psi^{(k-1)}(\lambda_{j}(1-a)+\mu_{j})\nonumber\\
&=&\sum_{k=1}^{n}\left(_{k}^{n}\right)\frac{(1-2a)^{k}}{(k-1)!}\sum_{j=1}^{r}\lambda_{j}^{k}\psi^{(k-1)}(\lambda_{j}(1-a)+\mu_{j}+1)\nonumber\\
&&\ +\ \sum_{k=1}^{n}\left(_{k}^{n}\right)\frac{(1-2a)^{k}}{(k-1)!}\sum_{j=1}^{r}\lambda_{j}^{k}(-1)^{k-1}(k-1)!(\lambda_{j}(1-a)+\mu_{j})^{-k}\nonumber\\
&=&(1-2a)n\sum_{j=1}^{r}\lambda_{j}\psi(\lambda_{j}(1-a)+\mu_{j}+1)\nonumber\\
&&\ +\ \sum_{k=2}^{n}\left(_{k}^{n}\right)\left(-(1-2a)\lambda_{j}\right)^{k}\sum_{j=1}^{r}
\zeta(k,\lambda_{j}(1-a)+\mu_{j}+1)\nonumber\\
&&\ +\ \sum_{j=1}^{r}\left(\frac{a\lambda_{j}+\mu_{j}}{\lambda_{j}(1-a)+\mu_{j}}\right)^{n}-r.
\end{eqnarray}
Then
\begin{eqnarray}\label{eq.arch1}
S_{Arch}&=&n(1-2a)\log Q_{F}+(1-2a)n\sum_{j=1}^{r}\lambda_{j}\psi(\lambda_{j}(1-a)+\mu_{j}+1)\nonumber\\
&&\ +\ \sum_{k=2}^{n}\left(_{k}^{n}\right)\left(-(1-2a)\lambda_{j}\right)^{k}\sum_{j=1}^{r}
\zeta(k,\lambda_{j}(1-a)+\mu_{j}+1)\nonumber\\
&&\ +\ \sum_{j=1}^{r}\left(\frac{a\lambda_{j}+\mu_{j}}{\lambda_{j}(1-a)+\mu_{j}}\right)^{n}-r.
\end{eqnarray}
Let
$$\sum_{j=1}^{r}\sum_{k=2}^{n}\left(_{k}^{n}\right)\left(-(1-2a)\lambda_{j}\right)^{k}
\zeta(k,\lambda_{j}(1-a)+\mu_{j}+1)=\sum_{j=1}^{r}S_{Arch}(n,r).$$
To estimate  $S_{Arch}(n,r)$, we follow the proof  of the analogous theorem proved in \cite[Theorem 2]{O-S-ram}. Calculus of residues implies that
\begin{eqnarray}\label{eq.res}S_{Arch}(n,r)&=&\frac{(-1)^{n}}{2\pi i}n!\int_{R}\frac{\Gamma(s-n)}{\Gamma(s+1)}\left((1-2a)\lambda_{j}\right)^{k}
\zeta(k,\lambda_{j}(1-a)+\mu_{j}+1)ds\nonumber\\
&=&\frac{(-1)^{n}}{2\pi i}n!\int_{R}J_{r}(s)ds,
\end{eqnarray}
where $R$ is a positively oriented rectangle with vertices $3/2 \pm i$ and $n + 1/2 \pm i$ containing
simple poles $s = 2, 3,...,n$ of $J_{r}(s)$. Residues are
easily found using the fact that $Res_{s=-n}\Gamma(s)=\frac{(-1)^{n}}{n!}$, hence, for $k=2,...,n$, we get
$$Res_{s=k}J_{r}(s)=\frac{1}{k!}\left((1-2a)\lambda_{j}\right)^{k}\zeta(k,\lambda_{j}(1-a)+\mu_{j}+1)\frac{(-1)^{n-k}}{(n-k)!}.$$
This justify the equality \eqref{eq.res}. The function $J_{r}(s)$ is uniformly bounded on the real segment joining
$n + 1/2$ and $e^{n}$, hence, the rectangle $R$ can be deformed to the line
$(e^{n}-i\infty,e^{n}+i\infty)$. Further singularities of the function $J_{r}(s)$ are a
simple pole at $s = 0$ and a pole $s = 1$ of order 2. Therefore, we get
\begin{eqnarray}S_{Arch}(n,r)&=&(-1)^{n-1}n!(Res_{s=0}J_{r}(s)+Res_{s=1}J_{r}(s))\nonumber\\
&&\ +\ O\left(n!\int_{-\infty}^{+\infty}\frac{\Gamma(e^{n}+it-n)}{\Gamma(e^{n}+it+1)}
\left((1-2a)\lambda_{j}\right)^{e^{n}}\zeta(e^{n},\lambda_{j}(1-a)+\mu_{j}+1)dt\right).\nonumber
\end{eqnarray}
Since $1-a\geq1$ and $Re(\mu_{j})\geq0$,  as $n\rightarrow\infty$, we have
$$\left|\left((1-2a)\lambda_{j}\right)^{e^{n}}\zeta(e^{n},\lambda_{j}(1-a)+\mu_{j}+1)\right|=o(1).$$
Using the functional equation for the gamma function with the reflection formula, we get
\begin{eqnarray}n!\int_{-\infty}^{+\infty}\left|\frac{\Gamma(e^{n}+it-n)}{\Gamma(e^{n}+it+1)}\right|dt&\leq&
2n!\int_{0}^{+\infty}\frac{dt}{((e^{n}-n)^{2}+t^{2})^{n/2}}\nonumber\\
&&\ =\ \frac{2n!}{(e^{n}-n)^{n-1}}\int_{0}^{+\infty}\frac{dt}{(1+t^{2})^{n/2}}\nonumber\\
&\leq&\frac{(n+1)^{n+1}e^{-(n+1)}}{(e^{n}-n)^{n-1}}\nonumber\\
&\leq&e^{-n}.\nonumber
\end{eqnarray}
Then
$$S_{Arch}(n,r)=(-1)^{n-1}n!(Res_{s=0}J_{r}(s)+Res_{s=1}J_{r}(s))+O(1).$$
Since $\zeta(0,x,)=-x+1/2$, we have
$$(-1)^{n-1}n!Res_{s=0}J_{r}(s)=-\zeta(0,\lambda_{j}(1-a)+\mu_{j}+1)=\lambda_{j}(1-a)+\mu_{j}+\frac{1}{2}.$$
Using Laurent series representations of the factors appearing in $J_{r}(s)$, we  find
{\small$$(-1)^{n-1}n!Res_{s=1}J_{r}(s)=n\left(\frac{1}{2}-a\right)\left(\frac{\Gamma'}{\Gamma}(n)-\frac{\Gamma'}{\Gamma}
(\lambda_{j}(1-a)+\mu_{j}+1) +\log\left(\frac{1}{2}-a\right)+\gamma-1\right).
$$}
The following formula
$$\frac{\Gamma'}{\Gamma}(n)=\log n -\frac{1}{2n}-\sum_{i=1}^{K}\frac{B_{2k}}{2k}n^{-2k}+O_{K}(n^{-1-2k}),$$
as $n\rightarrow+\infty$, yields to
\begin{eqnarray}(-1)^{n-1}n!Res_{s=1}J_{r}(s)&=&n\left(\frac{1}{2}-a\right)\log n \nonumber\\ &&\ + n\left(\frac{1}{2}-a\right)\left(-\frac{\Gamma'}{\Gamma}(\lambda_{j}(1-a)+\mu_{j}+1)+\log\left(\frac{1}{2}-a\right)+\gamma-1\right)\nonumber\\
&&\ -\ \frac{1}{4}(1-2a)-\left(\frac{1}{2}-a\right)\sum_{i=1}^{K}\frac{B_{2k}}{2k}n^{-2k}+O_{K}(n^{-1-2k}),\nonumber
\end{eqnarray}
as $n\rightarrow+\infty$. Now, from \eqref{eq.psi} and \eqref{eq.arch1}, we obtain
\begin{eqnarray}\label{eq.arch}
&&S_{Arch}=\left(\frac{1}{2}-a\right)d_{F}n\log n \nonumber\\
&&\ \ \ \ \ \ \ \ \ \ \ \ \ + \ \ n(1-2a)\left\{\frac{d_{F}}{2}\left[\gamma-1-\log(1-2a)\right]+\frac{1}{2}\log (\lambda Q_{F})^{2}+C_{F}(a)\right\}+O(1),\nonumber\\
\end{eqnarray}
where $$C_{F}(a)=\sum_{j=1}^{r}\sum_{k=0}^{+\infty}(a\lambda_{j})^{k}\zeta(k,\lambda_{j}+\mu_{j}),\ \ \lambda=\prod_{j=1}^{r}\lambda_{j}^{2\lambda_{j}}$$
 and $\gamma$ is the Euler constant. \\

 Now, the non-archimedean contribution can be written in the form
 $$S_{NArch}=m_{F}\left[2-\left(1-\frac{2a-1}{a}\right)^{n}-\left(1+\frac{2a-1}{1-a}\right)^{n}\right] + \sum_{k=1}^{n}\left(_{k}^{n}\right)(1-2a)^{k}\eta_{k-1}(F,1-a).$$
 A simple adaptation of the argument used in \cite[Lemma 4.4]{maz} yields to that, if the Riemann hypothesis holds for $F\in{\widetilde{{\mathcal S}}}$, then
$$S_{NArch}=O\left(\sqrt{n}\log n\right).$$
Finally, we obtain the same formula given in Theorem \ref{th.asymp}, for $a<0$, we have
\begin{eqnarray}
\lambda_{F}(n,a)&=&(1/2-a)d_{F}n\log n\nonumber\\
&&\ +\ n(1-2a)\left\{\frac{d_{F}}{2}\left[\gamma-1-\log(1-2a)\right]+\frac{1}{2}\log (\lambda Q_{F})^{2}+C_{F}(a)\right\}\nonumber\\&&\ +\ O\left(\sqrt{n}\log n\right).\nonumber
\end{eqnarray}

\section{Concluding Remarks}
\begin{itemize}
	\item Suppose that the Riemann hypothesis holds and $\lambda_{F}(n,a)$ are real numbers otherwise we take the real part. Then, we have
	$$\lambda_{F}(n,a)=2\sum_{j=1}^{+\infty}\left[1-\cos(n\theta_{j})\right]\geq0,$$
	where  $\theta_{j}=\arctan\left(\frac{(1-2a)\gamma_{j}}{\gamma_{j}^{2}-a^{2}+a-\frac{1}{4}}\right)$ with $\rho=\frac{1}{2}+i\gamma_{j}$. We can write it as
		$$\lambda_{F}(n,a)=2\int_{0}^{+\infty}\left[1-\cos(n\theta(\gamma))\right]dN_{F}(\gamma),$$
	where the lower limit just as well way be taken as $\gamma_{1}$. Integrating by parts, we get
	\begin{eqnarray}\lambda_{F}(n,a)&=&-2n\int_{0}^{+\infty}\sin(n\theta(\gamma))N_{F}(\gamma)d\gamma\nonumber\\&&\ +\ \left[2(1-\cos(n\theta(\gamma))N_{F}(\gamma)\right]_{0}^{+\infty}.\nonumber
	\end{eqnarray}
	We made similar computation as in \cite[Equations (1.4),...,(1.9)]{coffey}, we get \footnote{These results can be obtained directly as in \cite[\S 4]{OOM1} from the identity
	$$\log\xi_{F}\left(\frac{\rho+a-1}{\rho-a}\right)=\log\xi_{F}(0)+\sum_{n=1}^{+\infty}\lambda_{F}(-n,a)\frac{(z-(1-a))^{n}}{n}.$$}
\begin{eqnarray}&&\lambda_{F}(n,a)=32(1-2a)n\nonumber\\&&\ \ \ \ \ \ \ \times\ \int_{0}^{+\infty}\frac{\gamma N_{F}(\gamma)}{(4\gamma^{2}-4a^{2}+4a+1)^{2}}U_{n-1}\left(\frac{4\gamma^{2}-4a^{2}+4a-1}{4\gamma^{2}-4a^{2}+4a+1}\right)d\gamma,\nonumber
\end{eqnarray}
	where  $U_{n-1}$ are the Chebyschev polynomials of the second kind. Furthermore, using the following relation between the Chebyshev polynomials of the second kind and the
first kind: $$\int U_{n}(x) dx =\frac{1}{n+1}T_{n+1}(x),$$
 we obtain
 $$\lambda_{F}(n,a)=2(1-2a)\sum_{k=1}^{+\infty}\alpha_{k}\left[1-T_{n}\left(\frac{4\gamma_{k}^{2}-4a^{2}+4a-1}{4\gamma_{k}^{2}-4a^{2}+4a+1}\right)\right],$$
 where the $\alpha_{k}$ count the number of zeros with imaginary part $\gamma_{k}$ including the
multiplicities.
\item Using the integral formula, it is easy to refined the asymptotic formula for $\lambda_{F}(n,a)$ under the Riemann hypothesis (see for example Coffey paper \cite[Proposition 7]{coffey}.\\
	
\end{itemize}

The  author in \cite{maz1} or with Omar and Ouni in \cite{OOM1,OOM2} conjectured
that the Li coefficients for some $L$-functions are increasing.
Then, a natural question is to see whether it remains true
for the modified Li coefficients $\lambda_{F}(n,a)$ with $a<0$. Furthermore, it is an interesting question to study how the modified Li coefficients of $L$-functions in the Selberg class (or sub-class) are distributed as well as the sum

$$\sum_{n\leq x}\lambda_{F}(n,a)$$ and state a relation between a  partial Li criterion which
relates the partial Riemann hypothesis to the positivity of the modified Li coefficients $\lambda_{F}(n,a)$ with $a<1/2$ up to a certain order $T>0$. It is also possible to extended the argument used in \cite{maz1} to the Selberg
class (or sub-class) to prove that the first modified Li coefficients are increasing using
the Bell polynomials without assuming the Riemann hypothesis.\\

These problems will be considered in a
sequel to this article.\\

\noindent{\bf Acknowledgements.}
The author thanks Prof. Sergey  Sekatskii for suggesting the
theme of this paper, giving some valuable advices (mainly on Section \ref{sec.3}) and more informations about his papers \cite{sek1,sek2}. Also,  the author would like to express his sincere gratitude
to Prof. Lajla Smajlovic for  many valuable suggestions which increased the clarity of the
presentation.

\end{document}